\documentclass[a4paper,12pt,reqno]{amsart}
\usepackage{eurosym}
\usepackage{amsfonts}
\usepackage{amsmath,amsthm,amssymb}
\usepackage{graphicx, color}
\usepackage{cite}
\usepackage{float}
\nonstopmode \numberwithin{equation}{section}

\newtheorem{theorem}{Theorem}
\newtheorem{corollary}{Corollary}[section]

\allowdisplaybreaks

\begin{document}
\title{On Fractional Kinetic Equations $\mathtt{k}$-Struve Functions Based Solutions}
\author{K.S. Nisar, S.R. Mondal, F.B.M Belgacem}
\address{K. S. Nisar : Department of Mathematics, College of Arts and
Science, Prince Sattam bin Abdulaziz University, Wadi Aldawaser, Riyadh
region 11991, Saudi Arabia}
\email{ksnisar1@gmail.com, n.sooppy@psau.edu.sa }
\address{S. R. Mondal: Department of Mathematics and statistics,King Faisal University, 
Al Ahsa, Saudi Arabia}
\email{saiful786@gmail.com}

\address{F.B.M. Belgacem: Department of Mathematics,  Faculty of Basic Education, PAAET, 
Al-Ardhiya, Kuwait}
\email{fbmbelgacem@gmail.com}


\begin{abstract}
 In the present research article, we investigate solutions for fractional kinetic equations, involving $\mathtt{k}$-Struve functions, some of the salient properties of which we present.  The method used is Laplace transform based.  the methodology and results can be adopted and extended to various related fractional problems in mathematical physics.
\end{abstract}
\maketitle
\section{Introduction}
Introduced by Hermann Struve in 1882, complex $p$  indexed functions with the same name, denoted and defined by
\begin{equation}
H_{p}\left( x\right) :=\sum_{k=0}^{\infty }\frac{\left( -1\right) ^{k}}{%
\Gamma \left( k+3/2\right) \Gamma \left( k+p+\frac{3}{2}\right) }\left( 
\tfrac{x}{2}\right) ^{2k+p+1},  \label{Struve-1}
\end{equation}

turn out to be particular solutions of the non-homogeneous Bessel differential equations, given by, 
\begin{equation}
x^{2}y^{^{\prime \prime }}\left( x\right) +xy^{^{\prime }}\left( x\right)
+\left( x^{2}-p^{2}\right) y\left( x\right) =\frac{4\left( \frac{x}{2}%
\right) ^{p+1}}{\sqrt{\pi }\Gamma \left( p+1/2\right) }.  \label{bde-1}
\end{equation}
The homogeneous version of \eqref{bde-1} have Bessel functions of the first kind, denoted as $J_{p}(x)$, for solutions, which are finite at $x = 0$, when $p$ a positive fraction and all integers \cite{Belgacem2010}, while tend diverge for negative fractions,$p$. The Struve functions occur in certain areas of physics and applied mathematics, for example, in water-wave and surface-wave problems \cite{Ahmadi,Hirata}, as well as in problems on unsteady aerodynamics \cite{Shaw}. The Struve functions are also important in particle quantum dynamical studies of spin
decoherence \cite{Shao} and nanotubes \cite{Pedersen}. For more details about Struve functions, their generalizations and properties, the esteemed reader is invited to consider references, \cite{Orhan,Yagmur, Bhowmick, Bhowmick2, Kant, Singh1, Singh2, Singh3, Singh4, Singh5}.
Recently, Nisar et al. \cite{Nisar-Saiful} introduced and studied various properties of $\mathtt{k}$-Struve function $\mathtt{S}_{\nu,c}^{\mathtt{k}}$ defined by
\begin{equation}\label{k-Struve}
\mathtt{S}_{\nu,c}^{\mathtt{k}}(x):=\sum_{r=0}^{\infty}\frac{(-c)^r}
{\Gamma_{\mathtt{k}}(r\mathtt{k}+\nu+\frac{3\mathtt{k}}{2})\Gamma(r+\frac{3}{2})}
\left(\frac{x}{2}\right)^{2r+\frac{\nu}{\mathtt{k}}+1}.
\end{equation}
In this paper, we consider \eqref{k-Struve} to obtain the solution of the fractional kinetic equations. 
Our methodology herein is Laplace transform based, and we plan to extend our findings in a future work by using the Sumudu transform,\cite{Belgacem2006a,Belgacem2006b}.  
Fractional calculus is developed to large area of mathematics physics and other engineering applications \cite{Gupta,Gupta,Nisar,Saichev,Saxena1,Saxena2,Saxena3,Saxena4,Zaslavsky,Katatbeh-Belgacem} because of its importance and efficiency.
The fractional differential equation between a chemical reaction or a production scheme (such as in birth-death processes) was established and treated by Haubold and Mathai \cite{Haubold}, (also see \cite{Belgacem2003,Chaurasia,Gupta}). 

\section{Solution of generalized fractional Kinetic equations for $\mathtt{k}$-Struve function}
\label{Sec2}

Let the arbitrary reaction described by a time-dependent quantity $N=\left( N_{t}\right) $. The rate of change $\frac{dN}{dt}$ to be a balance between the destruction rate $\mathfrak{d}$ and the production rate $\mathfrak{p}$ of N, that is, $\frac{%
dN}{dt}=-\mathfrak{d}+\mathfrak{p}$. Generally, destruction and production depend on the quantity N itself, that is, 
\begin{equation}
\frac{dN}{dt}=-\mathfrak{d}\left( N_{t}\right) +\mathfrak{p}\left( N_{t}\right)
\label{eqn-6-Struve}
\end{equation}
where $N_{t}$ described by $N_{t}\left( t^{\ast }\right)
=N\left( t-t^{\ast }\right) ,t^{\ast }>0$
The special case of ($\ref{eqn-6-Struve}$) was found in \cite{Haubold} in the following form, 
\begin{equation}
\frac{dN_{i}}{dt}=-c_{i}N_{i}\left( t\right)  \label{eqn-7-Struve}
\end{equation}
with $N_{i}\left( t=0\right) =N_{0}$, which is the number of density of species $i$ at time $t=0$ and $c_{i}>0$.
The solution of ($\ref{eqn-7-Struve}$) is easily seen to be given by 
\begin{equation}
N_{i}\left( t\right)=N_{0}e^{-c_{i}t}  \label{eqn-7a-Struve}
\end{equation}
Integration gives another form of ($\ref{eqn-7-Struve}$) as follows: 
\begin{equation}
N\left( t\right) -N_{0}=-c.~_{0}D_{t}^{-1}N\left( t\right)
\label{eqn-8-Struve}
\end{equation}
where $_{0}D_{t}^{-1}$ is the special case of the Riemann-Liouville integral
operator and c is a constant. The fractional generalization of ($\ref{eqn-8-Struve}$) is given by \cite%
{Haubold} as: 
\begin{equation}
N\left( t\right) -N_{0}=-c_{0}^{\upsilon }D_{t}^{-\upsilon}N\left( t\right)
\label{eqn-9-Struve}
\end{equation}
where $_{0}D_{t}^{-\upsilon }$ defined as 
\begin{equation}
_{0}D_{t}^{-\upsilon }f\left( t\right) =\frac{1}{\Gamma \left( \upsilon
\right) }\int\limits_{0}^{t}\left( t-s\right) ^{\upsilon -1}f\left( s\right)
ds,\mathcal{R}\left( \upsilon \right) >0  \label{eqn-9a-Struve}
\end{equation}
Suppose that $f(t)$ is a real or complex valued function of the (time)
variable $t > 0$ and s is a real or complex parameter. The Laplace transform
of $f(t)$ is defined by 
\begin{equation}
F\left( p\right) =L\left[ f(t):p\right] =\int\limits_{0}^{\infty
}e^{-pt}f\left( t\right) dt,\text{ \ \ }\mathcal{R}\left( p\right) >0
\label{eqn-13-Struve}
\end{equation}
The Mittag-Leffler function $%
E_\alpha\left(z\right)$ (see \cite{Mittag}) is defined by 
\begin{equation}
E_{\alpha}\left( z\right) =\sum_{n=0}^{\infty }\frac{z^{n}}{\Gamma \left(
\alpha n+1 \right) }\text{ \ \ }\left(z,\alpha \in \mathbb{C}; |z|<0,%
\mathcal{R}\left( \alpha\right)>0\right) .  \label{Mittag1}
\end{equation}
and $E_{\alpha ,\beta }\left( x\right)$ defined by \cite{Wiman} as 
\begin{equation}
E_{\alpha ,\beta }\left( x\right) =\sum_{n=0}^{\infty }\frac{x^{n}}{\Gamma
\left( \alpha n+\beta \right) }.  \label{Mittag2}
\end{equation}

\begin{theorem}
\label{Th1}If $d>0,\upsilon >0, l, c, t\in \mathbb{C}$ and $l>-\frac{3}{2}k$ then the solution of generalized fractional kinetic equation%
\begin{equation}
N\left( t\right) -N_{0}~S_{l,c}^{\mathtt{k}}\left( t\right) :=-d^{\upsilon }\text{ }%
_{0}D_{t}^{-\upsilon }N\left( t\right) ,  \label{eqn-14-Struve}
\end{equation}
is given by the following formula%
\begin{equation}
N\left( t\right) =N_{0}\sum_{r=0}^{\infty }\frac{\left( -c\right) ^{r}\Gamma
\left( 2r+\frac{l}{\mathtt{k}}+2\right) }{\Gamma_{\mathtt{k}}\left( r\mathtt{k}+l+\frac{3}{2}\mathtt{k}\right) \Gamma \left(
r+\frac{3}{2}\right)}\left( \frac{t}{2}\right) ^{2r+\frac{l}{\mathtt{k}}+1}E_{v,2r+\frac{l}{\mathtt{k}}+2}\left(
-d^{\upsilon }t^{\upsilon }\right) .  \label{eqn-15-Struve}
\end{equation}
where $E_{v,2r+\frac{l}{\mathtt{k}}+2}\left( .\right) $ is the generalized Mittag-Leffler
function in \eqref{Mittag2}
\end{theorem}

\begin{proof}
We begin by recalling the Laplace transform of the Riemann-Liouville
fractional integral operator (see, e.g.,\cite{Erdelyi},\cite{Srivastava}): 
\begin{equation}
L\left\{D_{t}^{-\upsilon }f\left( t\right) ;p\right\} =p^{-\upsilon
}F\left( p\right)  \label{eqn-16-Struve}
\end{equation}
where $F\left( p\right) $ is defined in ($\ref{eqn-13-Struve}$). Now, Taking
the Laplace transform of both sides of ($\ref{eqn-14-Struve}$) and using \eqref{k-Struve} and ($\ref{eqn-16-Struve}$), we get 
\begin{align*}
&L\left[ N\left( t\right) ;p\right] =N_{0}L\left[ S_{l,c}^{\mathtt{k}}\left( t\right) ;p%
\right] -d^{\upsilon }L\left[D_{t}^{-\upsilon }N\left( t\right) ;p%
\right]\\
&N\left( p\right) =N_{0}\left( \int_{0}^{\infty }e^{-pt}\sum_{r=0}^{\infty }%
\frac{\left( -c\right) ^{r}}{\Gamma_{\mathtt{k}} \left( r\mathtt{k}+l+\frac{3}{2}\mathtt{k}\right) \Gamma
\left( r+\frac{3}{2}\right) }\left( \frac{t}{2}\right) ^{2r+\frac{l}{\mathtt{k}}+1}\right) dt
-d^{\upsilon }p^{-\upsilon }N\left( p\right)\\
&N\left( p\right) +d^{\upsilon }p^{-\upsilon }N\left( p\right)
=N_{0}\sum_{r=0}^{\infty }\frac{\left( -c\right) ^{r}2^{-\left(
2r+\frac{l}{\mathtt{k}}+1\right) }}{\Gamma_{\mathtt{k}}\left( r\mathtt{k}+l+\frac{3}{2}\mathtt{k}\right) \Gamma \left( r+%
\frac{3}{2}\right) }\int\limits_{0}^{\infty }e^{-pt}t^{2r+\frac{l}{\mathtt{k}}+1}dt \\
&=N_{0}\sum\limits_{r=0}^{\infty }\frac{\left( -c\right) ^{r}\left(
2\right) ^{-\left( 2r+\frac{l}{\mathtt{k}}+1\right) }}{\Gamma \left( r\mathtt{k}+l+\frac{3}{2}\mathtt{k}\right)
\Gamma \left( r+\frac{3}{2}\right) }\frac{\Gamma \left( 2r+\frac{l}{\mathtt{k}}+2\right) }{%
p^{2r+\frac{l}{\mathtt{k}}+2}}
\end{align*}%
\begin{align}
&N\left( p\right) =N_{0}\sum\limits_{r=0}^{\infty }\frac{\left( -c\right)
^{r}\left( 2\right) ^{-\left( 2r+\frac{l}{\mathtt{k}}+1\right) }\Gamma \left( 2r+\frac{l}{\mathtt{k}}+2\right) }{%
\Gamma \left( rk+l+\frac{3}{2}\mathtt{k}\right) \Gamma \left( r+\frac{3}{2}\right) }\notag \\
&\times \left\{ p^{-\left( 2r+\frac{l}{\mathtt{k}}+2\right) }\sum\limits_{s=0}^{\infty }\left(
1\right) _{s}\frac{\left[ -\left( \frac{p}{d}\right) ^{-\upsilon }\right] }{%
\left( s\right) !}\right\}  \label{eqn-17-Struve}
\end{align}

Taking inverse Laplace transform of ($\ref{eqn-17-Struve}$), and by using $%
L^{-1}\left\{ p^{-\upsilon }\right\} =\frac{t^{\upsilon -1}}{\Gamma \left(
\upsilon \right) },\mathcal{R}\left( \upsilon \right) >0$, we have%
\begin{eqnarray*}
L^{-1}\left( N\left( p\right) \right) &=&N_{0}\sum\limits_{r=0}^{\infty }%
\frac{\left( -c\right) ^{r}\left( 2\right) ^{-\left( 2r+\frac{l}{\mathtt{k}}+1\right) }\Gamma
\left( 2r+\frac{l}{\mathtt{k}}+2\right) }{\Gamma \left( rk+\frac{l}{\mathtt{k}}+\frac{3}{2}\mathtt{k}\right) \Gamma \left(
r+\frac{3}{2}\right) } \\
&&\times L^{-1}\left\{ \sum\limits_{s=0}^{\infty }\left( -1\right)
^{s}d^{\upsilon s}p^{-\left( 2r+\frac{l}{\mathtt{k}}+2+\upsilon s\right) }\right\} \\
N\left( t\right) &=&N_{0}\sum\limits_{r=0}^{\infty }\frac{\left( -c\right)
^{r}\left( 2\right) ^{-\left( 2r+\frac{l}{\mathtt{k}}+1\right) }\Gamma \left( 2r+\frac{l}{\mathtt{k}}+2\right) }{%
\Gamma \left( rk+\frac{l}{\mathtt{k}}+\frac{3}{2}\mathtt{k}\right) \Gamma \left( r+\frac{3}{2}\right) }
\\
&&\times \left\{ \sum\limits_{s=0}^{\infty }\left( -1\right) ^{s}d^{\upsilon
s}\frac{t^{\left( 2r+\frac{l}{\mathtt{k}}+1+\upsilon s\right) }}{\Gamma \left( \upsilon
s+2r+\frac{l}{\mathtt{k}}+2\right) }\right\} \\
&=&N_{0}\sum\limits_{r=0}^{\infty }\frac{\left( -c\right) ^{r}\Gamma \left(
2r+\frac{l}{\mathtt{k}}+2\right) }{\Gamma \left( rk+\frac{l}{\mathtt{k}}+\frac{3}{2}\right) \Gamma \left( r+%
\frac{3}{2}\right) }\left( \frac{t}{2}\right) ^{2r+\frac{l}{\mathtt{k}}+1} \\
&&\times \left\{ \sum\limits_{s=0}^{\infty }\left( -1\right) ^{s}d^{\upsilon
s}\frac{t^{\upsilon s}}{\Gamma \left( \upsilon s+2r+\frac{l}{\mathtt{k}}+2\right) }\right\}
\end{eqnarray*}%
which is the desired result.
\end{proof}

\begin{corollary}
If we put $\mathtt{k}=1$ in $\left( \ref{eqn-15-Struve}\right) $ then we get the solution of fractional kinetic equation involving classical Struve function as:\\
If $d>0,\upsilon >0, l, c, t\in \mathbb{C}$ and $l>-\frac{3}{2}$ then the solution of generalized fractional kinetic equation%
\begin{align*}
N\left( t\right) -N_{0}~S_{l,c}^{1}\left( t\right) :=-d^{\upsilon }\text{ }_{0}D_{t}^{-\upsilon }N\left( t\right) , 
\end{align*}
is given by
\begin{align*}
N\left( t\right) =N_{0}\sum_{r=0}^{\infty }\frac{\left( -c\right) ^{r}\Gamma
\left( 2r+l+2\right) }{\Gamma\left( r+l+\frac{3}{2}\right) \Gamma \left(
r+\frac{3}{2}\right)}\left( \frac{t}{2}\right) ^{2r+l+1}E_{v,2r+l+2}\left(
-d^{\upsilon }t^{\upsilon }\right) .  
\end{align*}
\end{corollary}

\begin{theorem}
\label{Th2}
If $d>0,\upsilon >0,c,l,t\in \mathbb{C}$ and $l>-\frac{3}{2}$, then the equation %
\begin{equation}
N\left( t\right) -N_{0}S_{l,c}^{\mathtt{k}}\left( d^{\upsilon }t^{\upsilon }\right)
=-d^{\upsilon }{}_{0}D_{t}^{-\upsilon }N\left( t\right) 
\label{eqn-18-Struve}
\end{equation}%
have the following solution %
\begin{align}
&N\left( t\right) =N_{0}\sum\limits_{r=0}^{\infty }\frac{\left( -c\right)
^{r}\Gamma \left( 2r\upsilon +\upsilon l+\upsilon +1\right) }{\Gamma_{k} \left( r\mathtt{k}+l+\frac{%
3}{2}\mathtt{k}\right) \Gamma \left( r+\frac{3}{2}\right) }\left( \frac{d^{\upsilon }%
}{2}\right) ^{2r+l+1}t^{2\upsilon r+l \upsilon+\upsilon}\notag\\
&\times E_{\upsilon,2r\upsilon+l\upsilon+\upsilon+1}\left(
-d^{\upsilon }t^{\upsilon }\right)   \label{eqn-19-Struve}
\end{align}%
where $E_{\upsilon,2r\upsilon +l\upsilon+\upsilon+1}\left( .\right) $ is the generalized
Mittag-Leffler function $\left( \ref{Mittag2}\right) $.
\end{theorem}

\begin{proof}
Theorem \ref{Th2} can be proved in parallel with the proof of Theorem \ref%
{Th1}. So the details of proofs are omitted.
\end{proof}

\begin{corollary}
By putting $\mathtt{k}=1$ in Theorem \eqref{Th2}, we get the solution of fractional kinetic equation involving classical Struve function:
If $d>0,\upsilon >0,c,l,t\in \mathbb{C}$ and $l>-\frac{3}{2}$, then the equation %
\begin{equation*}
N\left( t\right) -N_{0}S_{l,c}^{1}\left( d^{\upsilon }t^{\upsilon }\right)
=-d^{\upsilon }{}_{0}D_{t}^{-\upsilon }N\left( t\right) 
\end{equation*}%

\begin{align}
&N\left( t\right) =N_{0}\sum\limits_{r=0}^{\infty }\frac{\left( -c\right)
^{r}\Gamma \left( 2r\upsilon +\upsilon l+\upsilon +1\right) }{\Gamma \left( r+l+\frac{%
3}{2}\right) \Gamma \left( r+\frac{3}{2}\right) }\left( \frac{d^{\upsilon }%
}{2}\right) ^{2r+l+1}t^{2\upsilon r+l \upsilon+\upsilon}\notag\\
&\times E_{\upsilon, 2r\upsilon+l\upsilon+\upsilon+1}\left(
-d^{\upsilon }t^{\upsilon }\right)   \label{eqn-20-Struve}
\end{align}%
\end{corollary}

\begin{theorem}
\label{Th3}If $d>0,\upsilon >0,c,l,t\in \mathbb{C}$ ,$\mathfrak{a}\neq d$
and $l>-\frac{3}{2}$, then the solution of the following equation%
\begin{equation}
N\left( t\right) -N_{0}S_{l,c}^{\mathtt{k}}\left( d^{\upsilon }t^{\upsilon }\right) =-%
\mathfrak{a}^{\upsilon }{}_{0}D_{t}^{-\upsilon }N\left( t\right) 
\label{eqn-23-Struve}
\end{equation}
hold the formula%
\begin{eqnarray}
N\left( t\right) &=&N_{0}\sum\limits_{r=0}^{\infty }\frac{\left( -c\right)
^{r}\Gamma \left( 2r\upsilon +\upsilon l+\upsilon +1\right) }{\Gamma_{\mathtt{k}} \left(
r\mathtt{k}+l+\frac{b+2}{2}\right) \Gamma \left( r+\frac{3}{2}\right) }\left( \frac{%
d^{\upsilon }}{2}\right) ^{2r+l+1}  \notag \\
&&\times t^{\upsilon \left( 2r+l+1\right) }E_{\upsilon, 2r\upsilon +l\upsilon +\upsilon+1}
\left( -\mathfrak{a}^{\upsilon }t^{\upsilon }\right)
\label{eqn-24-Struve}
\end{eqnarray}
\end{theorem}

\begin{corollary}
\label{Cor3}
If we set $\mathtt{k}=1$ then \eqref{eqn-24-Struve} reduced as follows:
\begin{eqnarray}
N\left( t\right) &=&N_{0}\sum\limits_{r=0}^{\infty }\frac{\left( -c\right)
^{r}\Gamma \left( 2r\upsilon +\upsilon l+\upsilon +1\right) }{\Gamma \left(
r+l+\frac{b+2}{2}\right) \Gamma \left( r+\frac{3}{2}\right) }\left( \frac{%
d^{\upsilon }}{2}\right) ^{2r+l+1}  \notag \\
&&\times t^{\upsilon \left( 2r+l+1\right) }E_{\upsilon, 2r\upsilon +l\upsilon +\upsilon+1}
\left( -\mathfrak{a}^{\upsilon }t^{\upsilon }\right)
\end{eqnarray}
\end{corollary}

\section{Graphical interpretation}
In this section we plot the graphs of our solutions of the kinetic
equation, which is established in \eqref{eqn-15-Struve}. In
each graph, we gave four solutions of the results on the basis of
assigning different values to the parameters.In figures 1, we take 
$\mathtt{k}=1$ and $\upsilon=0.5,1,1.5, 2$. Similarly figures 2-4 are plotted
respectively by taking $\mathtt{k}=2,3,4$. Figures 5-6 are plotted by considering 
$\upsilon=0.1,0.2,0.3, 0.4$ and for different values of $\mathtt{k}$. 
In this graphs, all other parameters are fixed by 1. 
It is clear from these figures that $N_{t}>0$ for $t > 0$ and
$N_{t}>0$ is monotonic increasing function for $t\in (0,\infty)$. 
In this study, we choose first 50 terms of Mittag-Leffler function and first
50 terms of our solutions to plot the graphs. $N_{t}=0$ , when
$t>0$ and $N_{t}\rightarrow\infty$ when $t\rightarrow1$ for all values of the parameters.

%\begin{figure}[H]
	%\centering
		%\includegraphics{keq1_1.eps}
	%\caption{Solution \eqref{eqn-15-Struve} for $N(t),\mathtt{k}=1$}
	%\label{fig:keq1_1}
%\end{figure}
%
%\begin{figure}[H]
	%\centering
		%\includegraphics{keq2.eps}
	%\caption{Solution \eqref{eqn-15-Struve} for $N(t),\mathtt{k}=2$}
	%\label{fig:keq2}
%\end{figure}
%
%\begin{figure}[H]
	%\centering
		%\includegraphics{keq3.eps}
	%\caption{Solution \eqref{eqn-15-Struve} for $N(t),\mathtt{k}=3$}
	%\label{fig:keq3}
%\end{figure}
%
%\begin{figure}[H]
	%\centering
		%\includegraphics{k4.eps}
	%\caption{Solution \eqref{eqn-15-Struve} for $N(t),\mathtt{k}=4$}
	%\label{fig:k4}
%\end{figure}
%
%\begin{figure}[H]
	%\centering
		%\includegraphics{keq2_2.eps}
	%\caption{Solution \eqref{eqn-15-Struve} for $N(t),\mathtt{k}=2$}
	%\label{fig:keq2_2}
%\end{figure}
%
%\begin{figure}[H]
	%\centering
		%\includegraphics{keq3_2.eps}
	%\caption{Solution \eqref{eqn-15-Struve} for $N(t),\mathtt{k}=3$}
	%\label{fig:keq3_2}
%\end{figure}

\end{document}